\documentclass[12pt]{amsart}

\usepackage{amsfonts, amssymb, amsmath, amscd, amsthm, txfonts, graphicx, color, enumerate}

 \allowdisplaybreaks

\theoremstyle{plain}
\newtheorem{teo}{Theorem}[section]
\newtheorem{propo}[teo]{Proposition}
\newtheorem{lem}[teo]{Lemma}
\newtheorem{cor}[teo]{Corollary}

\theoremstyle{definition}
\newtheorem{defin}[teo]{Definition}
\newtheorem{ej}[teo]{Example}

\numberwithin{equation}{section}

\theoremstyle{remark}
\newtheorem{obs}[teo]{Remark}

\usepackage{tikz}

\begin{document}
\title{Openness for Anosov Families}
\author{Jeovanny de Jesus Muentes Acevedo}
\address{Instituto de Matem\'atica e 
Estat\'istica\\ 
Universidade de S\~ao Paulo\\
05508-090, Sao Paulo, Brazil}  \email{jeovanny@ime.usp.br}




\begin{abstract}
 Anosov families were introduced by A. Fisher and P. Arnoux motivated by generalizing the notion of   Anosov diffeomorphism defined on a compact Riemannian manifold. Roughly, an Anosov family is a two-sided sequence of diffeomorphisms (or non-stationary dynamical system)  with similar behavior to an Anosov diffeomorphisms.   We show that the set consisting of Anosov families is   an open subset of the set consisting of  two-sided sequences    of diffeomorphisms, which  is equipped with the strong topology (or Whitney topology). 
\end{abstract}


\maketitle

\section{Introduction}
The Anosov families  were introduced by P. Arnoux and A. Fisher in  \cite{alb}, motivated by generalizing the notion of   Anosov diffeomorphisms. Roughly, an Anosov family is a two-sided sequence of diffeomorphisms    $\textbf{\textit{f}}=(f_{i})_{i\in\mathbb{Z}}$ defined on  a two-sided sequence of compact Riemannian manifolds $(M_{i})_{i\in\mathbb{Z}}$, which has a similar behavior to an Anosov diffeomorphisms, that is, each  tangent bundle $TM_{i}$   has a splitting into two subbundles, called stable and unstable subbundles, where the elements in the stable subbundle are    contracted by  $D (f_{i+n-1}\circ\cdots\circ f_{i}) $ and the elements in the unstable subbundle are    contracted by   $D (f_{i-n}^{-1}\circ\cdots\circ f_{i-1}^{-1})$, for $n\geq 1$. The study   of   sequences of applications  is known in the literature with several different names: non-stationary dynamical systems, non-autonomous dynamical systems,  sequences of mappings, among other names (see \cite{alb}, \cite{Bakhtin}, \cite{Bakhtin2}, \cite{Jeo1}).

\medskip

 Other approaches dealing      sequences of   diffeomorphisms with hyperbolic  behavior   can be found   in \cite{Bakhtin}, \cite{Bakhtin2},  \cite{Mikko}, among other works.   One difference between the notion considered in this paper and the  considered in the works above  mentioned   is that  the   $f_{i}$'s of the  Anosov families do not necessarily  are    Anosov diffeomorphisms (see \cite{alb}, Example 3). Furthermore,  the $M_{i}$'s, although they are diffeomorphic, they are not necessarily isometric, thus, the hyperbolicity could be induced by the Riemannian metrics (see \cite{alb}, \cite{Jeo2}  for more detail).  

\medskip

Let \( \textbf{M}\) be the disjoint union of the $M_{i}$'s, for $i\in\mathbb{Z}$, and   $\mathcal{F}(\textbf{M})$   the set consisting of the families of $C^{1}$-diffeomorphisms on \textbf{M} 
equipped      with the \textit{strong topo\-logy} (see  Definition \ref{topol}).     We denote by $\mathcal{A}(\textbf{M})$ the subset of $\mathcal{F}(\textbf{M})$ consisting of Anosov families. Young in \cite{young} proved  that families consisting of    $C^{1+1}$ random small perturbations of an Anosov diffeomorphism of class $C^{2}$ are   Anosov
families (see Remark \ref{Young}).  
 The main goal of this paper, which is to prove  that $\mathcal{A}(\textbf{M})$  is open  in $\mathcal{F}(\textbf{M})$, is a generalization of this result, since, as we said, Anosov families do not necessarily consist of Anosov diffeomorphisms.   This fact  
 will be fundamental to prove  the structural stability of some elements in $\mathcal{A}(\textbf{M})$, considering the uniform conjugacies to be given  in   Definition \ref{definconjugacy} (see   \cite{Jeo3}).  The result in \cite{Jeo3}   generalizes Theorem 1.1 in \cite{Liu}, which proves the structural stability of random small perturbations of hyperbolic diffeomorphisms.

\medskip

In the next section we define the class of objects to be studied in this work. We define the   composition law for a two-sided sequence of diffeomorphisms, the strong topology    and a type of conjugations which work   for the class of     families of diffeomorphisms. Furthermore,  we introduce the notion of Anosov family and  we present some examples of such families.   In Section 3 we will see several properties that satisfy  the Anosov families.  
  It is important to keep fixed the Riemannian metric   on each $M_{i}$, since the notion of Anosov family depends on the Riemannian metric   (see \cite{alb}, Example 4).     
 Other examples and properties of Anosov families   can be found in \cite{alb},  \cite{Jeo2} and \cite{Jeo3}.  In Section 4 we will prove that each family close to an Anosov family satisfies the property of the invariant cones (see Lemma \ref{fred}).     
  This fact will be fundamental for showing  the openness  of Anosov families, which will be proved in Theorem \ref{teoremaprin}.

\section{Anosov Families: Definition,  Examples and Uniform Conjugacy}
Given a two-sided sequence of Riemannian manifolds   $M_{i}$   with Riemannian metric $\langle \cdot, \cdot\rangle_{i}$ for $i\in \mathbb{Z}$, consider the  \textit{disjoint union}  $$\textbf{M}=\coprod_{i\in \mathbb{Z}}{M_{i}}=\bigcup_{i\in \mathbb{Z}}{M_{i}\times{i}}.$$ The set $\textbf{M}$ will be called   \textit{total space} and the $M_{i}$ will be called  \textit{components}.   We  give the total space  $\textbf{M}$   the Riemannian metric $\langle \cdot, \cdot\rangle$    induced by  $\langle \cdot, \cdot\rangle_{i} $ setting 
\begin{equation}\label{metricariemannaian} 
\langle \cdot, \cdot\rangle|_{M_{i}}=\langle \cdot, \cdot\rangle_{i} \quad\text{ for }i\in \mathbb{Z},
\end{equation} 
and we will use the notation $(\textbf{M},\langle \cdot, \cdot\rangle)$ for point out that we are considering the Riemannian  metric given in  \eqref{metricariemannaian}.  
 We denote by $\Vert \cdot\Vert_{i}$ the induced norm by  $\langle\cdot,\cdot\rangle_{i}$ on $TM_{i}$ and we will take   $\Vert \cdot \Vert$ defined on  $\textbf{M}$  as    $\Vert \cdot\Vert|_{M_{i}}=\Vert \cdot\Vert_{i} $ for $i\in \mathbb{Z}$. If $d_{i}(\cdot,\cdot)$ is the metric on $M_{i}$ induced by $\langle \cdot, \cdot\rangle_{i}$, the total space is equipped with the metric
\begin{equation}\label{metricatotal} 
d(x,y)=   \begin{cases}
        \min\{1,d_{i}(x,y)\}  & \mbox{if } x,y\in M_{i} \\
        	1 & \mbox{if } x\in M_{i}, y\in M_{j} \text{ and }i\neq j. 
        \end{cases} 
\end{equation}

\begin{defin}\label{leidecomposicao} A  \textit{non-stationary dynamical system} (or \textit{n.s.d.s.}) $(\textbf{M},\langle\cdot,\cdot\rangle, \textbf{\textit{f}})$  is a map $\textbf{\textit{f}}:\textbf{M}\rightarrow \textbf{M}$  such that, for each $i\in\mathbb{Z}$, $\textbf{\textit{f}}|_{M_{i}}=f_{i}:M_{i}\rightarrow M_{i+1}$ is a  $C^{1}$-diffeomorphism. Sometimes we use the notation   $\textbf{\textit{f}}=(f_{i})_{i\in\mathbb{Z}}$. The $n$-th composition is defined    as   
\begin{equation*}
   \textbf{\textit{f}}_{ i} ^{n}:= \begin{cases}
  f_{i+n-1}\circ \cdots\circ f_{i}:M_{i}\rightarrow M_{i+n} & \mbox{if }n>0 \\
  f_{i-n}^{-1}\circ \cdots\circ f_{i-1}^{-1}:M_{i}\rightarrow M_{i-n} & \mbox{if }n<0 \\
	I_{i}:M_{i}\rightarrow M_{i} & \mbox{if }n=0,
        \end{cases}
\end{equation*}
where $I_{i}:M_{i}\rightarrow M_{i}$ is the identity on $M_{i}$ (see Figure \ref{Toros12}). 
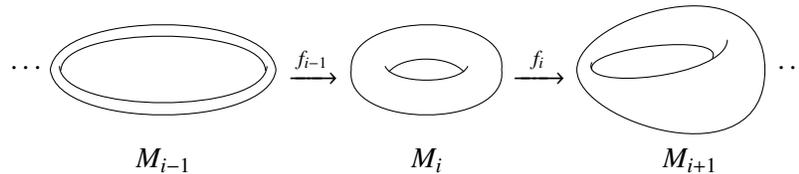
\begin{figure}[ht]
\begin{center}
\begin{tikzpicture}
\tikzstyle{point}=[circle,thick,draw=black,fill=black,inner sep=0pt,minimum width=1pt,minimum height=1pt]
\newcommand*{\xMin}{0}%
\newcommand*{\xMax}{6}%
\newcommand*{\yMin}{0}%
\newcommand*{\yMax}{6}%

\draw (-6.3,0.2) node[below] {\small \dots};
  
\draw (-6,0) .. controls (-5.9,0.8) and (-3.1,0.8) .. (-3,0);
\draw (-5.86,-0.01) .. controls (-5.85,0.6) and (-3.15,0.6) .. (-3.14,-0.01);
\draw (-6,0) .. controls (-5.9,-0.8) and (-3.1,-0.8) .. (-3,0);
\draw (-5.88,0.05) .. controls (-5.85,-0.6) and (-3.15,-0.6) .. (-3.12,0.05);
\draw (-4.5, -0.9) node[below] {\small $M_{i-1}$};

\draw (-2.5,0.45) node[below] {\small $\xrightarrow{f_{i-1}}$};

\draw (-2,0) .. controls (-2.05,0.8) and (0.1,0.8) .. (0,0);
\draw (-2,0) .. controls (-2.1,-0.8) and (0.1,-0.8) .. (0,0);
\draw (-1.5,-0.029) .. controls (-1.31,0.2) and (-0.7,0.2) .. (-0.5,-0.03);
\draw (-1.55,0.05) .. controls (-1.45,-0.2) and (-0.55,-0.2) .. (-0.45,0.05);

\draw (-1, -0.9) node[below] {\small $M_{i}$};

\draw (0.5,0.45) node[below] {\small  $\xrightarrow{\,\, \, f_{i}\, \, \,}$};

 \draw (1,0) .. controls (1,-0.7) and (3.5,-1.5) .. (3.5,0);
 \draw (1.2,0.1) .. controls (1,-0.3) and (3,-0.1) .. (3,0.4);

 \draw (1,0) .. controls (1,0.7) and (3.5,1.5) .. (3.5,0);
 \draw (1.2,0) .. controls (1.1,0.3) and (3,0.5) .. (2.8,0.15);
\draw (2.5, -0.9) node[below] {\small $M_{i+1}$};

 \draw (3.9,0.2) node[below] {\small \dots};
\end{tikzpicture}
\end{center}
\caption{A non-stationary dynamical system on a sequence of 2-torus endowed with different Riemannian metrics.}
\label{Toros12}
\end{figure} 
\end{defin}

One type of conjugacy that works  for the class of non-stationary dynamical systems is the \textit{uniform  conjugacy}:  
\begin{defin}\label{definconjugacy} A \textit{uniform  conjugacy} between two n.s.d.s. $  \textbf{\textit{f}}=(f_{i})_{i\in\mathbb{Z}}$ and $  \textbf{\textit{g}}=(g_{i})_{i\in\mathbb{Z}}$ on \textbf{M} is a map $\textbf{\textit{h}}:\textbf{M}\rightarrow \textbf{M}$, such that $\textbf{\textit{h}}|_{M_{i}}=h_{i}:M_{i}\rightarrow M_{i}$ is a  homeomorphism,    $(h_{i}:M_{i}\rightarrow M_{i})_{i\in\mathbb{Z}}$ and $(h_{i}^{-1}:M_{i}\rightarrow M_{i})_{i\in\mathbb{Z}}$ are  equicontinuous  families  and  $\textbf{\textit{h}}$  is a  \textit{topological  conjucacy} between the systems, i. e.,       \(h_{i+1}\circ f_{i}=g_{i}\circ h_{i}:M_{i}\rightarrow M_{i+1},\)   for every $ i\in \mathbb{Z}.$  This fact means that the following diagram  commutes: 
\[\begin{CD}
M_{-1}@>{f_{-1}}>> M_{0}@>{f_{0}}>>M_{1}@>{f_{1}}>>M_{2} \\
@V{\cdots}V{h_{-1}}V @VV{h_{0}}V @VV{h_{1}}V @VV{h_{2}\cdots}V\\
N_{-1}@>{g_{-1}}>> N_{0}@>{g_{0}}>>N_{1} @>{g_{1}}>>N_{2}
\end{CD} 
\]
In that case, we  will    say the families are \textit{uniformly conjugate}. 
\end{defin}

The
reason for   considering uniform    conjugacy instead of the topological conjugacy   is that   every n.s.d.s.  is topologically conjugate to the   n.s.d.s. whose maps are all the identity (see  \cite{alb}, Proposition 2.1).  Uniform conjugacies are also considered to characterize random dynamical systems (see \cite{Liu}). In \cite{Jeo1} we showed that the topological entropy for non-autonomous dynamical systems is a continuous map. The invariance of that entropy by uniform conjugacies is a fundamental tool to prove this result. 

\medskip
 
Consider     $\mathcal{F}(\textbf{M})=\{\textbf{\textit{f}}=(f_{i})_{i\in\mathbb{Z}}: f_{i}:M_{i}\rightarrow M_{i+1} \text{ is a  }C^{1}\text{-diffeomorphism}\}.$   
We endow $\mathcal{F}(\textbf{M})$  with the \textit{strong topology}:
\begin{defin}\label{topol}  
Let $\varepsilon=(\varepsilon_{i})_{i\in \mathbb{Z}}$ be a   sequence of positive numbers and  $\textbf{\textit{f}}\in \mathcal{F}(\textbf{M})$.  The set $$B(\textbf{\textit{f}},\varepsilon)=  \{\textbf{\textit{g}}\in \mathcal{F}(\textbf{M}): d_{\textbf{D}_{i}}(f_{i},g_{i})<\varepsilon_{i}\text{ for all }i\} $$ is called a \textit{strong basic neighborhood} of $\textbf{\textit{f}}$, where $d_{\textbf{D}_{i}}(\cdot,\cdot)$ is the  $C^{1}$-metric   on  $\textbf{D}_{i}= \text{Diff}^{1}(M_{i},M_{i+1})$, the set consisting of $C^{1}$-diffeomorphisms on $M_{i}$ to $M_{i+1}.$   The \textit{strong topology} (or \textit{Whitney topology}) is generated by the strong basic neighborhoods  of each $\textbf{\textit{f}}\in \mathcal{F}(\textbf{M})$.
  \end{defin}

\begin{defin}\label{estruturalmenteestavel} A subset  $\mathcal{A}$ of  $\mathcal{F}(\textbf{M})$  is open  if for each $\textbf{\textit{f}}\in \mathcal{A}$ there exists  $\varepsilon=(\varepsilon_{i})_{i\in \mathbb{Z}}$ such that $ B(\textbf{\textit{f}},\varepsilon)\subseteq \mathcal{A}$.  Furthermore, if for each $\textbf{\textit{f}}\in \mathcal{A}$ there is  $\varepsilon=(\varepsilon_{i})_{i\in \mathbb{Z}}$ such that, for any  $\textbf{\textit{g}}\in B(\textbf{\textit{f}},\varepsilon)$, \textbf{\textit{g}} is  uniformly  conjugate to \textbf{\textit{f}},  then we say that  $\mathcal{A}$ is  
\textit{structurally stable}.  
\end{defin}

\begin{defin}\label{anosovfamily}    A n.s.d.s \textbf{\textit{f}}    on \textbf{M}  is called an \textit{Anosov family} if:
\begin{enumerate}[i.]
\item the tangent bundle $T\textbf{M}$ has a continuous splitting   $E^{s}\oplus E^{u}$ which is  $D\textbf{\textit{f}}$-\textit{invariant}, i. e., for each $p\in \textbf{M}$, 
 $T_{p}\textbf{M}=E^{s}_{p}\oplus E^{u}_{p}$ with $D_{p}\textbf{\textit{f}}(E^{s}_{p})= E^{s}_{\textbf{\textit{f}}(p)}$ and $D_{p}\textbf{\textit{f}}(E^{u}_{p})= E^{u}_{\textbf{\textit{f}}(p)}$, where $T_{p}\textbf{M} $ is the    tangent space at $p;$
\item there exist constants $\lambda \in (0,1)$ and $c>0$ such that for each  $i\in \mathbb{Z}$, $n\geq 1$,    and $p\in M_{i}$, 
we have: \[\Vert D_{p}(\textbf{\textit{f}}_{i}^{n})(v)\Vert \leq c\lambda^{n}\Vert v\Vert \text{ if   }v\in E_{p}^{s}\quad\text{and}\quad \Vert D_{p}(\textbf{\textit{f}}_{i}^{-n}) (v)\Vert \leq c\lambda^{n}\Vert v\Vert  \text{ if }v\in E_{p}^{u}.\]
\end{enumerate}
The  subspaces $E^{s}_{p}$ and $E^{u}_{p}$ are called     stable and unstable subspaces, respectively.
\end{defin}
The set consisting of Anosov family on $(\textbf{M} ,\langle\cdot,\cdot\rangle)$ will be denoted by $\mathcal{A}(\textbf{M})$. 
If we can take
$c=1$ we say the family is \textit{strictly Anosov}.

 \medskip

A clear example of an   Anosov family is the \textit{constant family associated} to an  Anosov diffeomorphism   (see \cite{alb}, Definition 2.2).  
  Is well-known   the notion   of Anosov diffeomorphism  does not  depend  on the   Riemannian  metric   on the manifold (see \cite{Shub}).    However,   Example 4 in  \cite{alb} shows that  suitably changing the metric on each   $M_{i}$  the notion of Anosov family could not  be  satisfied.

\begin{ej} Let $F$ be a hyperbolic linear  cocycle   defined by $A:X\rightarrow SL(\mathbb{Z},d)$ over a homeomorphism $\phi:X\rightarrow X$ on a compact metric space $X$ (see \cite{Viana}). For each $x\in X$, the family $(A(f^{n}(x)))_{n\in\mathbb{Z}}$ defined on  $M_{i}=\mathbb{R}^{d}/\mathbb{Z}^{d}$, the torus $d$-dimensional equipped with the Riemannian metric inherited from $\mathbb{R}^{d}$, determines an Anosov family. 
\end{ej}

\begin{obs}\label{Young} Let $\phi:M\rightarrow M$ be an  Anosov diffeomorphism  of class $C^{2}$ on a compact Riemannian manifold   
$M$  and  $\beta >0$ such that  $L(D\phi)<\beta$,  where $L(D\phi)$ is a Lipchitz constant of the derivative application $x\mapsto D_{x}\phi$. For $\alpha>0$, take $$\Omega_{\alpha,\beta}(\phi)=\{\psi\in C^{1}(M):  d(\phi,\psi)\leq \alpha \text{ and } L(D\psi)\leq \beta\},$$
 where $d(\cdot,\cdot)$ is the    $C^{1}$-metric on  $\text{Diff}^{1}(M)$.  If $\alpha $ is small enough,  any sequence $(\psi_{i})_{i\in \mathbb{Z}}$ in $\Omega_{\alpha,\beta}(\phi)$ defines  an   Anosov family in $\textbf{M}=\coprod_{i\in \mathbb{Z}}{M}$ 
(see  \cite{young}, Proposition 2.2).
Consequently,    the set consisting of   the constant families  associated to   Anosov diffeomorphisms of class $C^{2}$ is open in $\mathcal{F}(\textbf{M})$. \end{obs}   
Using the above fact we have:
 
\begin{ej} Given $\alpha\in \mathbb{R}$, consider    $\phi_{\alpha}:\mathbb{T}^{2}\rightarrow \mathbb{T}^{2}$ defined by  \[\phi_{\alpha}(x,y)=(2x+y-(1+\alpha)\sin x  \text{ mod } 2\pi,x+y-(1+\alpha)\sin x \text{ mod } 2\pi ).\]
For all $\alpha\in [-1,0),$  $\phi_{\alpha}$ is  an Anosov diffeomorphism (see    \cite{luisb}). 
We have that given $\alpha^{\star} \in [-1,0)$  there exists   $\varepsilon>0$ such that, if $(\alpha_{i})_{i\in\mathbb{Z}}$ is a sequence in $[-1,0)$ with $|\alpha_{i}-\alpha^{\star}|<\varepsilon$, then $(f_{i})_{i\in\mathbb{Z}}$ is an  Anosov family, where $f_{i}=\phi_{\alpha_{i}}$ for   $i\in \mathbb{Z}$. 
\end{ej}
  
 The existence of Anosov diffeomorphisms $\phi:M\rightarrow M$ imposes strong restrictions on the manifold $M$.  All known examples of Anosov diffeomorphisms are defined on \textit{infranilmanifolds} (see      
 \cite{luisb}, \cite{Shub},  \cite{Viana}). The circle $\mathbb{S}^{1}=\{x\in\mathbb{R}^{2}:\Vert x\Vert=1\}$ does not admit any Anosov diffeomorphism.  In  \cite{Jeo2} we show that   $\textbf{S}^{1}$ does not admit  Anosov families in the following sense: let $\textbf{M}=\bigcup_{i\in\mathbb{Z}}M_{i}$ where $M_{i}=\mathbb{S}^{1}\times\{i\}$   equipped with the Riemannian metric inherited from $\mathbb{R}^{2}$ for each $i$.  Thus, there is not    any Anosov family on $\textbf{M}$. 
As mentioned above, the Anosov families are not necessarily formed by Anosov diffeomorphisms. Then, a natural question that arises from the notion of Anosov families is: which compact Riemannian manifolds admit Anosov families?

\section{Some Properties of the    Anosov Families}
 
We now show some properties that the Anosov families  satisfy and that will be used in the rest of the work. 
In this section,  if we do not say otherwise,  $(\textbf{M},\langle\cdot,\cdot\rangle,\textbf{\textit{f}})$ will represent an  Anosov family with constants $\lambda\in (0,1)$ and $c\geq1$.
Sometimes  we will omit the index  $i$ of $f_{i}$ if it is clear that we are considering the $i$-th diffeomorphism of   \textbf{\textit{f}}. 

 \medskip
 
 In \cite{alb}, Proposition 2.12, is shown for an Anosov family the splitting $T_{p}\textbf{M}=E_{p}^{s}\oplus E_{p}^{u}$ is unique.   Actually, we have:
 \begin{lem}\label{unicidadesubs} For each $p\in M_{i}$  we have
\begin{enumerate}[i.]
\item $E_{p} ^{s}=\{v\in T_{p}M_{i}:  \Vert D_{p}(
\textbf{\textit{f}}^{\, n})  (v)\Vert \text{ is bounded, for }n\geq 1\}.$
\item $E_{p} ^{u}=\{v\in T_{p}M_{i}:  \Vert D_{p}(\textbf{\textit{f}}^{-n})  (v)\Vert \text{ is bounded, for }n\geq 1\}.$
\end{enumerate}
\end{lem}
\begin{proof} We  will prove   i.  Set  $B_{p} ^{s}=\{v\in T_{p}M_{i}:  \sup_{n\geq1}\Vert D_{p}(\textbf{\textit{f}}^{n})  (v)\Vert <+\infty\}.$  It is clear that $E_{p}^{s}\subseteq B_{p} ^{s}$. Suppose   there exists     $v \in T_{p}M_{i}$ such that $v\notin E ^{s}_{p}.$ Thus $v =v_{s}+v_{u}$, for some $v_{s}\in E ^{s}_{p}$ and  $v_{u}\in E ^{u}_{p}$ with $v_{u}\neq0$. Therefore, we have   
$\Vert D_{p}(\textbf{\textit{f}}^{n})  (v)\Vert  \geq      c^{-1}\lambda^{-n}\Vert v_{u}\Vert - c \lambda^{n}\Vert v_{s}\Vert ,$
where $\Vert D_{p}(\textbf{\textit{f}}^{n}) (v )\Vert\rightarrow +\infty$, that is, $v\notin B_{p} ^{s}$. Thus $ B^{s}_{p}\subseteq E^{s}_{p}.$  
\end{proof}

\begin{defin}\label{conesestaveiseinstaveis1}
For $p\in \textbf{M}$ and $  \alpha >0$, set 
\begin{align*} 
K_{\alpha,\textbf{\textit{f}},p}^{s} &=\{(v_{s},v_{u})\in E_{p}^{s}\oplus E_{p}^{u}: \Vert v_{u}\Vert< \alpha \Vert v_{s} \Vert\}\cup\{(0,0)\}:=  \textit{stable } \alpha\textit{-cone of }  \textbf{\textit{f}} \textit{ at}  p,  \\
K_{\alpha,\textbf{\textit{f}},p}^{u} &=\{(v_{s},v_{u})\in E_{p}^{s}\oplus E_{p}^{u}: \Vert v_{s}\Vert< \alpha \Vert v_{u}\Vert\}\cup\{(0,0)\}:= \textit{unstable }   \alpha \textit{-cone of } \textbf{\textit{f}} \text{ at}  p.
\end{align*}
(see Figure \ref{coneses1}). 
\begin{figure}[ht] 
\begin{center}

\begin{tikzpicture}
\draw[black!7,fill=black!17, ultra thin] (-2.4,-0.4)rectangle (1.4,3.4);
\draw[black, fill=black!47, very thin] (-0.5,1.5) -- (-1.5,3.4) -- (0.5,3.4) -- cycle;
\draw[black, fill=black!47, very thin] (-0.5,1.5) -- (-1.5,-0.4) -- (0.5,-0.4) -- cycle;
\draw[black, fill=black!47, very thin] (-0.5,1.5) -- (1.4,2.5) -- (1.4,0.5) -- cycle;
\draw[black, fill=black!47, very thin] (-0.5,1.5) -- (-2.4,0.5) -- (-2.4,2.5) -- cycle;
 
 \draw[<->] (-0.5,-0.5) -- (-0.5,3.5);
 \draw[<->] (-2.5,1.5) -- (1.5,1.5);
\draw (-0.7,4) node[below] {\quad{\scriptsize $E_{p}^{u}$}};
\draw (1.6,1.7) node[below] {\quad{\scriptsize $E_{p}^{s}$}};
\draw (1.8,3.5) node[below] {\scriptsize $T_{p}M$};
\draw (0.2,3.9) node[below] {\quad{\scriptsize $K_{\alpha,\textbf{\textit{f}},p}^{u}$}};
\draw (1.7,2.3) node[below] {\quad{\scriptsize $K_{\alpha,\textbf{\textit{f}},p}^{s}$}};
 \end{tikzpicture} 

\end{center}
\caption{Stable and unstable $\alpha$-cones     at $p$.}
\label{coneses1}
\end{figure}
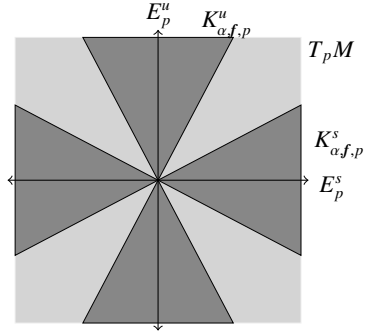 
\end{defin} 

 Taking a suitable $\alpha$, the following lemma shows that the cones are invariant by the derivative of    the family and, in addition, the derivative  of the   family restricted  $K_{\alpha,\textbf{\textit{f}},p}^{u}$ is an  expansion  and restricted to $K_{\alpha,\textbf{\textit{f}},p}^{s}$ is a contraction:

\begin{lem}\label{lemacones} Suppose that $\textbf{\textit{f}}$ is a strictly Anosov family. Fix   $\alpha\in (0,\frac{1-\lambda}{1+\lambda})$ and take $\lambda^{\prime}=\lambda\frac{1+\alpha}{1-\alpha}<1$. Thus:
\begin{enumerate}[i.]
\item $ D_{p}\textbf{\textit{f}}(K_{\alpha,\textbf{\textit{f}},p}^{u})\subseteq K_{\alpha,\textbf{\textit{f}},\textbf{\textit{f}}(p)}^{u}$. Furthermore,   $\Vert D_{p}\textbf{\textit{f}}(v)\Vert \geq \frac{1}{\lambda^{\prime}}  \Vert v\Vert$ for   $v\in  K_{\alpha,\textbf{\textit{f}},p}^{u}$.
\item $ D_{\textbf{\textit{f}}(p)}\textbf{\textit{f}}^{-1}(K_{\alpha,\textbf{\textit{f}},\textbf{\textit{f}}(p)}^{s})\subseteq K_{\alpha,\textbf{\textit{f}},p}^{s}$. Furthermore,    $\Vert D_{\textbf{\textit{f}}(p)}\textbf{\textit{f}}^{-1}(v)\Vert \geq \frac{1}{\lambda^{\prime}}  \Vert v\Vert$  for $v\in  K_{\alpha,\textbf{\textit{f}},\textbf{\textit{f}}(p)}^{s}$.
 \end{enumerate}
\end{lem}
\begin{proof} 
For   $(v_{s},v_{u})\in  K_{\alpha,\textbf{\textit{f}},p}^{u}$ we have  \[\Vert D_{p}\textbf{\textit{f}}(v_{s})\Vert \leq \lambda \Vert v_{s}\Vert \leq \lambda\alpha \Vert v_{u}\Vert \leq \lambda^{2}\alpha\Vert D_{p}\textbf{\textit{f}}(v_{u})\Vert\leq \alpha\Vert D_{p}\textbf{\textit{f}}(v_{u})\Vert. \]
Therefore $ D_{p}\textbf{\textit{f}}(K_{\alpha,\textbf{\textit{f}},p}^{u})\subseteq K_{\alpha,\textbf{\textit{f}},\textbf{\textit{f}}(p)}^{u}$. On the other hand,  we have
\begin{equation*}\Vert D_{p}\textbf{\textit{f}}(v_{s},v_{u})\Vert \geq \Vert D_{p}\textbf{\textit{f}}(v_{u})\Vert- \Vert D_{p}\textbf{\textit{f}}(v_{s})\Vert \geq\frac{1-\alpha}{\lambda(1+\alpha)} \Vert(v_{s},v_{u})\Vert, 
\end{equation*}
and this fact proves i. 
The part ii. can be proved analogously.
\end{proof}

Next proposition proves    the continuity of the splitting     $E^{s}\oplus E^{u}$ can be obtained  from both the condition ii. in  
	Definition \ref{anosovfamily} and   the $D\textbf{\textit{f}}$-invariance of the splitting.    We adapt the ideas of the proof  of   Proposition 2.2.9 in  \cite{luisb} (which is done for diffeomorphisms defined on compact Riemannian manifolds) to show the following result.

\begin{propo}\label{continuidade} Let $  \textbf{\textit{f}}\in\mathcal{F}(\textbf{M})$. Suppose that $T\textbf{M}$ has a splitting     $E^{s}\oplus E^{u}$ which is  $D\textbf{\textit{f}}$-invariant and satisfies the  property ii. from   Definition \ref{anosovfamily}. Thus,  $E_{p}^{s} $ and $E_{p}^{u} $ depend continuously on $p$. 
\end{propo}
 
\begin{proof} First we prove that the dimensions of   $E^{u}$ and  $E^{s}$ are  locally constants. Let $p\in \textbf{M}$ and    $k=\text{dim} E_{p} ^{s}$.  Suppose by contradiction that there exists a sequence $(p_{m})_{m\in\mathbb{N}}\subseteq \textbf{M}$ converging to $p$ such that $\text{dim} E_{p_{m}} ^{s}\geq k+1$ for all $m$. Take a sequence  of orthonormal  vectors   \[v_{1}(p_{m}),...,v_{k}(p_{m}) , v_{k+1}(p_{m})\quad\text{ in }E_{p_{m}} ^{s},\text{ for each }m.\] Choosing a suitable subsequence, we can suppose that \[v_{1}(p_{m})\rightarrow v_{1}\in T_{p}\textbf{M},...,v_{k+1}(p_{m})\rightarrow v_{k+1}\in T_{p}\textbf{M}\text{ \quad as }m\rightarrow \infty.\]  Therefore, by continuity of the Riemannian metric, it follows from condition ii. in   Definition \ref{anosovfamily} that, for all $n\geq 1,$ we have
   \begin{equation}\label{dddddd}\Vert D_{p}(\textbf{\textit{f}}_{i}^{n}) (v_{s})\Vert \leq c\lambda^{n}\Vert v_{s}\Vert\quad\text{for each }s=1,\dots,k+1.
   \end{equation}
By Lemma \ref{unicidadesubs} we obtain    $v_ {1},... , v_{k+1}\in E_{p}^{s}.$
Since    $v_{1}(p_{m}),...,v_{k}(p_{m}) ,$ and $ v_{k+1}(p_{m})$ are orthonormal for all $m\geq 1$, we have that  $v_ {1},... , v_{k+1}$ are  orthonormal, which contradicts that $\text{dim} E_{p} ^{s}=k$. Similarly we can prove that  there is not       any sequence   $(p_{m})_{m\in\mathbb{N}}$ converging to $p$ with $\text{dim} E_{p_{m}} ^{s}<k$ for all $m$. Therefore, the dimension of  $E_{p} ^{s}$ is  locally  constant.

Analogously we obtain that the dimension of  $E_{p} ^{u}$ is locally constant.

\medskip

Now, let $(p_{m})_{m\in\mathbb{N}}$ be a sequence in $\textbf{M}$ such that $p_{m}\rightarrow p\in \textbf{M}$ as $m\rightarrow \infty.$ Without loss of generality, we can suppose that $(p_{m})_{m\in\mathbb{N}}\subseteq M_{i}$ and $p\in M_{i}$ for some $i\in\mathbb{Z}$. This fact  follows from the definition of the metric on \textbf{M} given  in  \eqref{metricatotal}. Furthermore, we can assume that $\text{dim}  E_{p_{m}}^{s} = \text{dim}  E_{p}^{s} =k$ for every $m\geq 1$. Let     $\{v_{1}(p_{m}),...,v_{k}(p_{m})  \}$ be an  orthonormal basis of $E_{p_{m}}^{s} $, for each $m\geq 1$, such that $v_{1}(p_{m})\rightarrow v_{1}\in T_{p}M_{i},...,v_{k}(p_{m})\rightarrow v_{k}\in T_{p}M_{i} $ as $m\rightarrow \infty$. By the continuity of the Riemannian metric we have that  $ v_{1} ,...,v_{k} $ are orthonormal and   \begin{equation*}\Vert D_{p}(\textbf{\textit{f}}_{i}^{n}) (v_{s})\Vert \leq c\lambda^{n}\Vert v_{s}\Vert\quad\text{for each }s=1,\dots,k, 
   \end{equation*} 
 that is, $v_{1} ,...,v_{k} $ belong to $E_{p}^{s} $. This fact proves that $E_{p}^{s}$ depends continuously on $p$. Analogously we can prove that  $E_{p}^{u}$ depends continuously on $p$. 
 \end{proof}

The notion   of Anosov diffeomorphism  does not  depend of the   Riemannian  metric on the manifold (see \cite{Shub}). In contrast, the notion of Anosov family depends on the Riemannian metric taken on each $M_{i}$ (see \cite{alb}, Example  4).    However, the next proposition   proves that  the notion of    Anosov family does not depend on the Riemannian metric chosen  uniformly equivalent on \textbf{M}.\footnote{Two Riemannian metrics $\langle\cdot,\cdot\rangle$ and  $\langle\cdot,\cdot\rangle_{\ast}$ defined on a manifold $M$ are \textit{uniformly equivalent} if there exist positive numbers $k$ and $K$ such that $k\langle v,v\rangle\leq \langle v,v\rangle_{\ast}\leq K\langle v,v\rangle$ for any $v\in TM$.}

\begin{propo}\label{normasequiv} Let  $\langle\cdot,\cdot\rangle$ and $\langle\cdot,\cdot\rangle^{\star}$   be  Riemannian metrics   uniformly equivalent  on $\textbf{M}$.  We have that     $(\textbf{M},\langle\cdot,\cdot\rangle ,\textbf{\textit{f}})$  is an   Anosov family if, and only if,   $(\textbf{M},\langle\cdot,\cdot\rangle^{\star},\textbf{\textit{f}})$  is an Anosov family.\end{propo} 
\begin{proof}  Let $\Vert \cdot\Vert$ and $\Vert \cdot\Vert^{\star}$ be the norms induced by $\langle\cdot,\cdot\rangle$ and $\langle\cdot,\cdot\rangle^{\star}$, respectively. Since $\langle\cdot,\cdot\rangle$ and $\langle\cdot,\cdot\rangle^{\star}$ are     uniformly equivalent  on $\textbf{M}$, there exist $k>0$ and $K>0$  such that 
\(k\Vert v\Vert ^{\star}\leq \Vert v\Vert \leq K\Vert v\Vert ^{\star} \) for all \(v\in T\textbf{M}  .\) 
Suppose that  $(\textbf{M},\langle\cdot,\cdot\rangle ,\textbf{\textit{f}})$  is an   Anosov family with constant $\lambda\in(0,1)$ and $c\geq1$. Thus,     for  $v\in T_{p}\textbf{M}, n\geq1$, 
\[\Vert D_{p}(\textbf{\textit{f}}_{i} ^{n})(v)\Vert ^{\star}\leq  (1/k)\Vert D_{p}(\textbf{\textit{f}}_{i} ^{n})(v)\Vert \leq (c/k)\lambda^{n}\Vert v\Vert\leq (Kc/k)\lambda^{n}\Vert v\Vert^{\star} .\]
Analogously we have  \(\Vert D_{p}(\textbf{\textit{f}}_{i} ^{-n})(v)\Vert ^{\star}\leq    (Kc/k)\lambda^{n}\Vert v\Vert^{\star},\)   for \(v\in T_{p}\textbf{M}, n\geq1.\)
Therefore,  $(\textbf{M},\langle\cdot,\cdot\rangle^{\star},\textbf{\textit{f}})$  is an Anosov family with constant $\lambda$ and $\tilde{c}=Kc/k.$

Similarly   we  can prove   if $(\textbf{M},\langle\cdot,\cdot\rangle^{\star} ,\textbf{\textit{f}})$  is an   Anosov family then  $(\textbf{M},\langle\cdot,\cdot\rangle,\textbf{\textit{f}})$  is an Anosov family.
\end{proof}

  In    Proposition  \ref{mather} we will  show  there exists a Riemannian metric   $\langle\cdot,\cdot\rangle^{\star}$,   equivalent to  $\langle\cdot,\cdot\rangle$ on each  $M_{i}$ ($\langle\cdot,\cdot\rangle^{\star}$   is not necessarily uniformly equivalent to $\langle\cdot,\cdot\rangle$  on the total space \textbf{M}), with which, $(\textbf{M},\langle\cdot,\cdot\rangle^{\star},\textbf{\textit{f}})$ is a  strictly Anosov family. That is a version for families of a well-known Lemma of Mather for Anosov diffeomorphisms  (see \cite{Shub}). In order to prove this fact,  we introduce the following notion: 
Fix  $i\in \mathbb{Z}$. Since for each $p\in M_{i}$, the subspaces $E^{s}_{p}$ and $E^{u}_{p}$ are  transversal, that is, $E^{s}_{p}\oplus E^{u}_{p}=T_{p} M_{i}$,  then, by the compactness of  $M_{i}$ and the continuity of both the Riemannian metric and the subspaces  $E^{s}_{p}$ and  $E^{u}_{p}$, we obtain that there exists  $\mu_{i}\in (0,1)$  such that, if $v_{s}$ and $v_{u}$ are unit vectors in $ E_{p}^{s}$ and $ E_{p}^{u}$, respectively, then \begin{equation}\label{propang} \text{cos}(\widehat{v_{s}v_{u}})\in [\mu_{i} -1,1-\mu_{i}],
\end{equation}
where $\widehat{v_{s}v_{u}}$ is the  angle between $v_{s}$ and $v_{u}.$     In the case of Anosov diffeomorphisms defined on compact manifolds  those angles  are uniformly bounded away from 0. In    \cite{Jeo2} we  gave  an example where the  angles between the unstable and stable subspaces along the orbit of a point of $M_{0}$ converge to zero. 

 \begin{defin}\label{propang2} We say that an Anosov family satisfies the \textit{property of the angles} (or \textit{s. p. a.}) if there exists $ \mu\in(0,1)$  such that, for all $p\in\mathbf{M}$, if $v_{s}\in E_{p}^{s}$ and  $v_{u}\in E_{p}^{u}$,  then
$\cos(\widehat{v_{s}v_{u}})\in [\mu -1,1-\mu],$ 
 that is, $\mu$ does not depend on $i$. \end{defin}

\begin{propo}\label{mather} There exists a  $C^{\infty}$ Riemannian metric $\langle\cdot,\cdot\rangle^{\star}$ on \textbf{M}, which is
   uniformly    equivalent to  $\langle\cdot,\cdot\rangle $ on each $M_{i}$,  such that  $(\textbf{M},\langle\cdot,\cdot\rangle^{\star},\textbf{\textit{f}}\,)$ is a strictly  Anosov family. Furthermore,  $(\textbf{M},\langle\cdot,\cdot\rangle^{\star},\textbf{\textit{f}}\,)$  satisfies the property of the angles.
\end{propo}
\begin{proof} Let $\varepsilon\in (0,1-\lambda)$.   For  $p\in \textbf{M}$, if $(v_{s},v_{u})\in E_{p}^{s}\oplus E_{p}^{u},$ take  \begin{equation}\label{equivnorm}\Vert (v_{s},v_{u})\Vert_{1}=\sqrt{{\Vert v_{s}\Vert_{1}} ^{2}+{\Vert v_{u}\Vert_{1}}^{2}}, 
\end{equation} where
$ \Vert v_{s}\Vert_{1} =\sum_{n=0} ^{\infty} (\lambda+\varepsilon)^{-n}\Vert D_{p}(\textbf{\textit{f}} ^{\,n}) v_{s}\Vert $ and $\Vert v_{u}\Vert_{1} =\sum_{n=0} ^{\infty} (\lambda+\varepsilon)^{-n}\Vert D_{p}(\textbf{\textit{f}}^{\, -n}) v_{u}\Vert.$
Note that if $v_{s}\in E_{p}^{s}$ we have
\begin{equation}\label{qwerty} 
\Vert v_{s}\Vert_{1} =\sum_{n=0} ^{\infty} (\lambda+\varepsilon)^{-n}\Vert D_{p}(\textbf{\textit{f}}^{n}) v_{s}\Vert\leq \sum_{n=0} ^{\infty} (\lambda+\varepsilon)^{-n}c\lambda^{n}\Vert  v_{s}\Vert  =\frac{\lambda+\varepsilon}{\varepsilon}c\Vert v_{s}\Vert.
\end{equation}
 Analogously, $\Vert v_{u}\Vert_{1}\leq \frac{\lambda+\varepsilon}{\varepsilon}c\Vert v_{u}\Vert$ for  $v_{u}\in E_{p}^{u}.$ Consequently the series  $\Vert v_{s}\Vert_{1}$ and  $\Vert v_{u}\Vert_{1}$  converge uniformly. That is, $\Vert \cdot\Vert_{1}$ is well defined.

  We prove  that  $\Vert \cdot\Vert_{1}$ is  uniformly equivalent to  $\Vert \cdot\Vert$ on each  $M_{i}$.      It is clear that $\Vert v_{s}\Vert \leq \Vert v_{s}\Vert_{1}$ and $\Vert v_{u}\Vert \leq \Vert v_{u}\Vert_{1}.$ Thus, 
\begin{align*}\Vert(v_{s},v_{u})\Vert  &\leq \Vert v_{s}\Vert+\Vert v_{u}\Vert \leq 2(\Vert v_{s}\Vert^{2}+\Vert v_{u}\Vert^{2})^{1/2}\leq 2(\Vert v_{s}\Vert^{2}_{1}+\Vert v_{u}\Vert^{2}_{1})^{1/2}  =2\Vert (v_{s},v_{u})\Vert _{1}.
\end{align*}
This fact implies  \begin{equation}\label{norma2312}\Vert v\Vert \leq 2\Vert v\Vert _{1}\quad\text{ for all }v\in T\textbf{M}.
\end{equation} 

 Fix $p\in M_{i}$.   Let  $\theta_{p}$ be the angle between two vectors     $v_{s}\in E_{p} ^{s}$ and $v_{u}\in E_{p} ^{u}$, for  $p\in M_{i}$. Take $\mu_{i}$ as in  \eqref{propang}.   Since $(1-\mu_{i})(\Vert v_{s}\Vert^{2}+\Vert v_{u}\Vert ^{2})\geq 2(1-\mu_{i})\Vert v_{s}\Vert \Vert v_{u}\Vert$, we have $$\Vert v_{s}\Vert^{2}+\Vert v_{u}\Vert ^{2}+2(\mu_{i}-1)\Vert v_{s}\Vert \Vert v_{u}\Vert\geq \mu_{i}(\Vert v_{s}\Vert^{2}+\Vert v_{u}\Vert ^{2}).$$
 Therefore 
\begin{align*}\Vert (v_{s},v_{u})\Vert ^{2}
&=\Vert v_{s}\Vert^{2}+ \Vert v_{u}\Vert ^{2}-2\text{cos} \theta_{p}\Vert v_{s}\Vert\Vert v_{u}\Vert \geq \Vert v_{s}\Vert^{2}+ \Vert v_{u}\Vert ^{2}+2(\mu_{i}-1)\Vert v_{s}\Vert\Vert v_{u}\Vert \\
&
\geq \mu _{i}(\Vert v_{s}\Vert^{2}+\Vert v_{u}\Vert ^{2}).
\end{align*} 
Consequently, \begin{align*}
\Vert (v_{s},v_{u})\Vert_{1}^{2}  &=
\Vert v_{s}\Vert_{1} ^{2}+\Vert v_{u}\Vert_{1} ^{2}\leq (\frac{\lambda+\varepsilon}{\varepsilon}c)^{2}(\Vert v_{s}\Vert^{2}+\Vert v_{u}\Vert^{2}) \leq \frac{1}{\mu_{i}}(\frac{\lambda+\varepsilon}{\varepsilon}c)^{2}\Vert (v_{s},v_{u})\Vert ^{2}.
\end{align*} 
Thus, 
   \begin{equation}\label{norma12}\Vert v\Vert_{1}\leq \frac{1}{\mu_{i}}(\frac{\lambda+\varepsilon}{\varepsilon}c)^{2}\Vert v\Vert\quad\text{ for all }v\in TM_{i}.
   \end{equation}

It follows from \eqref{norma2312} and  \eqref{norma12} that   \begin{equation}\label{norma231}\frac{1}{2}\Vert v\Vert \leq \Vert v\Vert _{1}\leq  \frac{1}{\mu_{i}}(\frac{\lambda+\varepsilon}{\varepsilon}c)^{2}\Vert v\Vert\quad\text{ for all }v\in TM_{i}.
\end{equation}  
Hence, the norm $\Vert \cdot\Vert_{1}$ is  uniformly equivalent to the norm $\Vert \cdot\Vert$ on each  $M_{i}$. 
 
We have also   that  \begin{equation*}\Vert D_{p}\textbf{\textit{f}}v_{s}\Vert_{1}  
  \leq (\lambda+\varepsilon)\Vert v_{s}\Vert_{1}\text{ if }v_{s}\in E_{p} ^{s} \text{ \, and \, }\Vert D_{p}(\textbf{\textit{f}} ^{\, -1})v_{u}\Vert_{1}  
 \leq (\lambda+\varepsilon)\Vert v_{u}\Vert_{1}\text{ if }v_{u}\in E_{p} ^{u}.
 \end{equation*}   
  Note that the norm $\Vert\cdot\Vert_{1}$ comes from an inner product $\langle\cdot,\cdot\rangle_{1} $,  which defines a continuous Riemannian metric on \textbf{M}. Consequently, for each $i$, we can choose a $C^{\infty}$-Riemannian metric $\langle\cdot,\cdot\rangle^{\star}_{i}$  such that $|\langle v,v\rangle^{\star}_{i}-\langle v,v\rangle_{1}|<\varepsilon$ for each $v\in TM_{i}$. We take $\langle \cdot,\cdot\rangle^{\star}$ on $\textbf{M}$,  defined on each $M_{i}$ as $\langle \cdot,\cdot\rangle^{\star}|_{M_{i}}=\langle \cdot,\cdot\rangle^{\star}_{i}.$
Hence $(\textbf{M},\langle\cdot,\cdot\rangle^{\star},\textbf{\textit{f}})$  is a strictly Anosov family with constant   $\lambda^{\prime}=\lambda+\varepsilon $, which s.   p.   a.. \end{proof}

By \eqref{norma231} we have that   $\langle\cdot,\cdot\rangle  $ and $\langle\cdot,\cdot\rangle^{\star}$ are uniformly equivalent on each $M_{i}$. However, this fact does not imply that they are uniformly equivalent on \textbf{M}, because $\mu_{i}$ could converge to 0 as $i\rightarrow \pm \infty$ (notice that \textbf{M} is not compact).   If the angles between the stable and unstable subspaces converge to zero along   an orbit, then $\mu_{i}$ converges to zero. In that case the two metrics are not uniformly equivalent on the total space. On the other hand:  

\begin{cor}\label{mather2} If   $(\textbf{M},\langle\cdot,\cdot\rangle,\textbf{\textit{f}})$  s.  p.   a., then there exists a  $C^{\infty}$-Riemannian metric  $\langle\cdot,\cdot\rangle^{\star}$,       uniformly  equivalent to $\langle\cdot,\cdot\rangle$ on \textbf{M},  such that $(\textbf{M},\langle\cdot,\cdot\rangle^{\star},\textbf{\textit{f}})$ is a strictly Anosov family that s.  p.   a..   
\end{cor}  
\begin{proof}  Since $\textbf{\textit{f}}$ satisfies the property of the angles, we can take a $\mu$ as in  Definition \ref{propang2}. From \eqref{norma231}   we have for all \(v\in T\textbf{M},\) \[\frac{1}{2}\Vert v\Vert \leq \Vert v\Vert _{1}\leq  \frac{1}{\mu}(\frac{\lambda+\varepsilon}{\varepsilon}c)^{2}\Vert v\Vert,\] where $\Vert \cdot\Vert_{1}$ is the metric defined in \eqref{equivnorm}.
Thus, $\Vert \cdot\Vert$ and $\Vert\cdot\Vert_{1}$ are uniformly equivalent on the total space. The  corollary follows from the proof  of   Proposition  \ref{mather}.\end{proof}

 A Riemannian metric is \textit{adapted to an hyperbolic set} of a diffeomorphism if, in this
metric, the expansion (contraction) of the unstable (stable) subspaces is seen after only one
iteration.   The metric obtained in    Proposition \ref{mather} is adapted to \textbf{M} for the family \textbf{\textit{f}}.  This metric is not always   uniformly equivalent to $\langle\cdot,\cdot\rangle$, because there exist Anosov families which do not s.  p.   a..

 \section{Invariant Cones}

    In order to prove the openness of $\mathcal{A}(\textbf{M})$, we use the method of the invariant cones (see \cite{luisb}).  We  will prove that  there exists a strong basic neighborhood $B(\textbf{\textit{f}},(\varepsilon_{i})_{i\in\mathbb{Z}})$ of \textbf{\textit{f}} such that   each family in $B(\textbf{\textit{f}},(\varepsilon_{i})_{i\in\mathbb{Z}})$ satisfies Lemma \ref{lemacones}.
    
    \medskip
      
 We will use the  exponential application  to work on a  Euclidian  ambient space.  For each $i\in \mathbb{Z}$, there exists $\delta_{i}>0$ such that, if  $p\in M_{i}$, then the exponencial application  at $p$, $\text{exp}_{p}:B_{p}(0,\delta_{i})\rightarrow B(p,\delta_{i})$,  is a diffeomorphism,    and $ \Vert v\Vert =d(\text{exp}_{p}(v), p),$ for all  $v\in B_{p}(0,\delta_{i}),$  where $B_{p}(0,\delta_{i})$ is the ball in $T_{p}M_{i}$ with radius   $\delta_{i}$ and center $0\in T_{p}M_{i}$   and $B(p,\delta_{i})$ is the ball in $M_{i}$  with  radius $\delta_{i}$ and center $p$, i.e.,  $\delta_{i}$ is the \textit{injectivity radius} of the exponential application at each $p\in M_{i}$.   The injectivity radius   could decrease as $ |i|$ increases, since the $M_{i}$'s are different.  We need a radius   small enough such that the inequality in \eqref{sigmaadecuado}   be valid. This inequality  depends also on the behavior of each $f_{i}$.
 
 \medskip
 
 By simplicity, in this section we will suppose that  $\textbf{\textit{f}}\in \mathcal{F}(\textbf{M})$ is an  Anosov family that satisfies the property of the angles.     
 
\begin{obs}We can choose    $\beta_{i}>0$, with $\beta_{i}<\min\{\delta_{i-1},\delta_{i},\delta_{i+1}\}/2$, such that, if $p\in M_{i}$,  
$f(B(p,2\beta_{i}))\subseteq B(f (p),\delta_{i+1}/2)$  and $ f^{-1}(B(f(p),2\beta_{i+1}))\subseteq B(p,\delta_{i}/2).$ 
Thus, if  $\textbf{\textit{g}}=(g_{i})_{i\in\mathbb{Z}}\in \mathcal{F}(\textbf{M})$ with $d_{\textbf{D}_{i}}(f_{i},g_{i})<\beta_{i}$ for all $i$,  we have
\begin{equation}\label{betan}g(B(p,\beta_{i}))\subseteq B(f (p),\delta_{i+1})\quad\text{and}\quad g^{-1}(B(f(p),\beta_{i+1}))\subseteq B(p,\delta_{i}).
\end{equation} 
\end{obs}

 Consider a linear isomorphism   $\tau_{p}:T_{p}\textbf{M}\rightarrow \mathbb{R}^{d}$, depending continuously on $p$,  which maps an orthonormal basis  of $E^{s}_{p}$ to an orthonormal basis of $\mathbb{R}^{k}$ 
 and maps an orthonormal basis  of $E^{u}_{p}$ to an orthonormal basis of $\mathbb{R}^{d-k}$, where $d$ is the dimension of each $M_{i}$ 
and $k$  the dimension of    $E^{s}_{p}$.  Since \textbf{\textit{f}} satisfies the property of the angles, the norm $\Vert \cdot\Vert_{1}$ 
defined in \eqref{equivnorm} is uniformly  equivalent to the norm    $\Vert \cdot \Vert$ (Corollary \ref{mather2}).
Hence, without loss of generality, we can suppose that $\Vert\cdot\Vert=\Vert\cdot\Vert_{1}$, because a  family of diffeomorphisms  in any strong basic neighborhood of \textbf{\textit{f}} is Anosov with $\Vert\cdot\Vert$ if and only if is Anosov with $\Vert\cdot\Vert_{1}$ (see Proposition \ref{normasequiv}).   Therefore, we can suppose that   \textbf{\textit{f}}  is   strictly Anosov.  Note that    $\Vert \tau_{p}(v)\Vert  =\Vert v\Vert$ for all $v\in T_{p}\textbf{M}.$  
  
  \medskip

For $g \in \textbf{D}_{i} $, with $d_{\textbf{D}_{i}}(f_{i},g )<\beta_{i}$, we set \begin{align*}\tilde{g}_{p}&=\tau_{f(p)}\circ\text{exp}_{f(p)}^{-1}\circ g_{i}\circ\text{exp}_{p}\circ\tau_{p} ^{-1}:B_{p}(0,\beta_{i})\rightarrow B_{f(p)}(0,\delta_{i+1})\\ 
\text{and }\quad\tilde{g}_{p}^{-1}&=\tau_{p}\circ\text{exp}_{p}^{-1}\circ g_{i} ^{-1}\circ\text{exp}_{f(p)}\circ\tau_{f(p)} ^{-1}:B_{f(p)}(0,\beta_{i+1})\rightarrow B_{p}(0,\delta_{i})
,\end{align*}
which are well-defined as a consequence of \eqref{betan}.

\begin{defin}
Let $B^{k}(0,\beta_{i})\subseteq \mathbb{R}^{k}$ and  $ B^{d-k}(0,\beta_{i} )\subseteq \mathbb{R}^{d-k}$ be the open balls with center at   $0$ and radius $\beta_{i}$. 
For $x\in \mathbb{R}^{d}$, we denote by $(x)_{1} $ and $(x)_{2}$ the orthogonal projections of $x$ on $E^{s}$ and $E^{u}$, respectively. If   $(v,w)\in B^{k}(0,\beta_{i} )\times B^{d-k}(0,\beta_{i} )$, then \begin{align*}\tilde{f}_{p}(v,w)&=((\tilde{f_{p}})_{1}(v,w),(\tilde{f_{p}})_{2}(v,w))=(\tilde{a}_{p}(v,w)+\tilde{F}_{p}(v) ,\tilde{b}_{p}(v,w) +\tilde{F}_{p}(w)),
\end{align*}
where $\tilde{a}_{p}(v,w)=(\tilde{f_{p}})_{1}(v,w)-\tilde{F}_{p}(v),$ $ 
\tilde{b}_{p}(v,w)=(\tilde{f_{p}})_{2}(v,w)-\tilde{F}_{p}(w), $ and $ 
\tilde{F}_{p} =D_{0}(\tilde{f}_{p})   .$
Analogously we have that, for each   $(v,w)\in B^{k}(0,\beta_{i+1} )\times B^{d-k}(0,\beta_{i+1} )$,   \begin{equation*}\tilde{f}_{p}^{-1}(v,w)=(\tilde{c}_{p}(v,w)+\tilde{G}_{p}(v) ,\tilde{d}_{p}(v,w) +\tilde{G}_{p}(w)),
\end{equation*}
with \(\tilde{c}_{p}(v,w)=(\tilde{f}_{p}^{-1})_{1}(v,w)- \tilde{G}_{p} (v)\); \(
 \tilde{d}_{p}(v,w)=(\tilde{f}_{p}^{-1})_{2}(v,w)-\tilde{G}_{p}(w)\); \( 
 \tilde{G}_{p} =D_{0}(\tilde{f}_{p}^{-1}).\)
\end{defin}
 
Consider   \begin{align*}\sigma_{1,p}&= \sup \{\Vert D_{(v,w)}(\tilde{a}_{p},\tilde{b}_{p})\Vert: (v,w)\in B^{k}(0,\beta_{i} )\times B^{d-k}(0,\beta_{i})\}\\
    \text{and} \quad   \sigma_{2,p}&= \sup \{\Vert D_{(v,w)}(\tilde{c}_{p},\tilde{d}_{p})\Vert: (v,w)\in B^{k}(0,\beta_{i+1} )\times B^{d-k}(0,\beta_{i+1})\}.
    \end{align*} 
Note that $\sigma_{1,p}$ and $\sigma_{2,p}$ depend  on $\beta_{i}$. Take  \(\sigma_{p}=\max\{\sigma_{1,p},\sigma_{2,p}\} .\)

\begin{lem}\label{primeirolem} Fix $\alpha\in (0,\frac{1-\lambda}{1+\lambda})$. For each $i\in \mathbb{Z}$ there exists $\beta_{i}$ such that \begin{equation}\label{sigmaadecuado}
\sigma_{i}:=\max_{p\in M_{i}}\sigma_{p}\leq \min\left\{ \frac{(\lambda^{-1}-\lambda)\alpha}{2(1+\alpha)^{2}},\frac{\lambda^{-1}(1-\alpha)-(1+\alpha)\alpha}{2(1+\alpha)}\right\}.
\end{equation}
\end{lem}
\begin{proof} 
Note that $D_{0}(\tilde{f}_{p})=\tau_{f(p)}D_{p}f\tau_{p}^{-1}$. Hence, if $(v,w)\in \mathbb{R}^{k} \oplus \mathbb{R}^{d-k}$, we have
\begin{align*}(\tilde{F}_{p}v,\tilde{F}_{p}w)&=(\tau_{f(p)}D_{p}f\tau_{p}^{-1}(v),\tau_{f(p)}D_{p}f\tau_{p}^{-1}(w))=\tau_{f(p)}D_{p}f\tau_{p}^{-1}(v,w)\\
&=D_{0}(\tilde{f}_{p})(v,w)
=(D_{0}(\tilde{f}_{p})_{1}(v,w),D_{0}(\tilde{f}_{p})_{2}(v,w)).
\end{align*}
Consequently,   
$D_{0}(\tilde{a}_{p})=0$ and $ D_{0}(\tilde{b}_{p})=0.$ 
 Analogously, we can prove that   $D_{0}(\tilde{c}_{p})=0$ and $ D_{0}(\tilde{d}_{p})=0.$ 
Thus, since    $f$ is of class $C^{1}$ and   $M_{i}$ is compact, it follows that for each  $i$ we can choose  $\beta_{i}$ small enough such that   \eqref{sigmaadecuado} is valid.
\end{proof}
 
   We chose $\alpha \in (0, \frac{1- \lambda}{1+ \lambda})$ for the minimum in \eqref{sigmaadecuado} be positive.  Set 
 \begin{align*}K_{\alpha}^{s}& =\{(v,w)\in \mathbb{R}^{k}\oplus \mathbb{R}^{d-k}: \Vert w\Vert< \alpha \Vert v \Vert\};\\
K_{\alpha}^{u}  &=\{(v,w)\in \mathbb{R}^{k}\oplus \mathbb{R}^{d-k}: \Vert v\Vert< \alpha \Vert w\Vert\}.
\end{align*}  

\begin{lem}\label{limitadocones12}  Let  $\alpha\in (0,\frac{1-\lambda}{1+\lambda})$ and $\beta_{i}$  be as in    Lemma \eqref{primeirolem}. Thus, there exists a $\varepsilon_{i}>0$   such that,  if  $g\in \textbf{D}_{i}$  with $d_{\textbf{D}_{i}}(f_{i},g)<\varepsilon_{i}$, for all $p\in M_{i}$ we have: 
\begin{enumerate}[i.]
\item  $D_{(v,w)}\tilde{g}_{p}
(\overline{K_{\alpha}^{u} })
\subseteq K_{\alpha}^{u}$ for all   $(v,w)\in B^{k}(0,\beta_{i} )\times B^{d-k}(0,\beta_{i} )$, and
\item $
 D_{(v,w)}\tilde{g}_{p} ^{-1}(\overline{K_{\alpha}^{s}})
\subseteq K_{\alpha}^{s}$ for all   $(v,w)\in B^{k}(0,\beta_{i+1} )\times B^{d-k}(0,\beta_{i+1} )$.
\end{enumerate}
\end{lem}
\begin{proof} We will prove   i.  Take  
$ \varepsilon_{i} < \min\{\beta_{i},\beta_{i+1},\sigma_{i} \}
 .$ 
Fix $(v,w)\in  B^{k}(0,\beta_{i} )\times B^{d-k}(0,\beta_{i} )$. If $(x,y)\in \overline{K_{\alpha}^{u}} \setminus\{(0,0)\}$, then   
 \begin{align*}\Vert (D_{(v,w)}\tilde{g}_{p}(x,y))_{1}\Vert&\leq \Vert (D_{(v,w)}\tilde{g}_{p}(x,y))_{1}-(D_{(v,w)}\tilde{f}_{p}(x,y))_{1}\Vert +\Vert (D_{(v,w)}\tilde{f}_{p}(x,y))_{1}\Vert\\
& \leq  \sigma_{i} (\alpha\Vert y\Vert +\Vert y\Vert)+\sigma_{i}\Vert (x,y)\Vert+\lambda\Vert x\Vert \leq  ((\alpha+1)2\sigma_{i} +\lambda\alpha)\Vert y\Vert.
\end{align*}
Analogously, we have 
\(\Vert (D_{(v,w)}\tilde{g}_{p}(x,y))_{2}\Vert \geq(\lambda^{-1} -2\sigma_{i}(\alpha+1))\Vert y\Vert.\)
Since $\sigma_{i} <\frac{\alpha( 
 \lambda^{-1}-\lambda)}{2(1+\alpha)^{2}}$,  then   $\frac{(\alpha+1)2\sigma_{i}+\lambda \alpha}{\lambda^{-1}-2\sigma_{i}(\alpha+1)} < \alpha ,$ and hence,  $\Vert  (D_{(v,w)}\tilde{g}_{p}(x,y))_{1}\Vert    
 <\alpha \Vert (D_{(v,w)}\tilde{g}_{p}(x,y))_{2}\Vert.$
Therefore,   $D_{(v,w)}\tilde{g}_{p}(x,y)\in K_{\alpha}^{u}.$ Consequently,  $D_{(v,w)}\tilde{g}_{p}
(\overline{K_{\alpha}^{u} })
\subseteq K_{\alpha}^{u}$. 
\end{proof}

\begin{lem}\label{limitadocontract12} If $\varepsilon_{i}<\min\{\beta_{i},\beta_{i+1},\sigma_{i}\}$, there exists $\eta<1$ such that, 
if $g\in \textbf{D}_{i}$ is such that  $d_{\textbf{D}_{i}}(f_{i},g)<\varepsilon_{i},$ then,  for $p\in M_{i},$
\begin{enumerate}[i.]
\item $\Vert D_{(v,w)}\tilde{g}_{p}(x,y)\Vert \geq \eta^{-1}\Vert (x,y) \Vert$ if $(x,y)\in \overline{K_{\alpha}^{u}}$;
\item $\Vert   D_{(v,w)}\tilde{g}^{-1} _{p}(x,y)\Vert\geq \eta^{-1}\Vert (x,y)\Vert$ if $(x,y)\in \overline{K_{\alpha}^{s}}$.
\end{enumerate}
\end{lem}
\begin{proof} We will prove i.    Let $g\in \textbf{D}_{i}$ be such that   $d_{\textbf{D}_{i}}(f_{i},g)<\varepsilon_{i}$. Fix    $p\in M_{i}$ and take
 $ (x,y)\in \overline{K_{\alpha}^{u}}$. By Lemma \ref{limitadocones12} we have  $\Vert (D_{(v,w)}\tilde{f}_{p}(x,y))_{1}\Vert\leq \alpha\Vert  (D_{(v,w)}\tilde{f}_{p}(x,y))_{2}\Vert
$ for $(v,w)\in B^{k}(0,\beta_{i})\times B^{d-k}(0,\beta_{i})$. Thus, 
\begin{align*}
\Vert D_{(v,w)}\tilde{g}_{p}(x,y)\Vert &\geq \Vert D_{(v,w)}\tilde{f}_{p}(x,y)\Vert   -\Vert D_{(v,w)}\tilde{f}_{p}(x,y) -D_{(v,w)}\tilde{g}_{p}(x,y)\Vert\\
&\geq  \Vert (D_{(v,w)}\tilde{f}_{p}(x,y))_{2}\Vert-\Vert (D_{(v,w)}\tilde{f}_{p}(x,y))_{1}\Vert-\varepsilon_{i}\Vert (x,y)\Vert\\
&\geq  (1-\alpha)(\Vert \tilde{F}_{p}(y)\Vert -\Vert D_{(v,w)}\tilde{b}_{p}(x,y)\Vert)   -\sigma_{i}\Vert (x,y)\Vert\\
&\geq  (1-\alpha)(\frac{\lambda^{-1}}{1+\alpha}\Vert (x,y)\Vert -\sigma_{i}\Vert (x,y)\Vert)   -\sigma_{i}\Vert (x,y)\Vert.
\end{align*}
Consequently, \(\Vert D_{(v,w)}\tilde{g}_{p}(x,y)\Vert\geq 
   \frac{1}{\eta}\Vert (x,y)\Vert,\)
where $\frac{1}{\eta}:=(1-\alpha)(\frac{\lambda^{-1}}{1+\alpha} -\sigma_{i}) -\sigma_{i}$. 
  Since $\sigma_{i} < \frac{(1-\alpha)\lambda^{-1}-(1+\alpha)}{2(1+\alpha)},$  \(\eta<1\).   
\end{proof}

  Fix   $\textbf{\textit{g}}=(g_{i})_{i\in\mathbb{Z}}\in B(\textbf{\textit{f}}, (\varepsilon_{i})_{i\in\mathbb{Z}}).$
For each $i\in\mathbb{Z}$, let $m_i\in \mathbb{N}$  be such that  $M_{i}=\cup _{j=1}^{m_{j}}B(p_{j,i},\beta_{i})$, where $p_{j,i}\in M_{i}$, for  $j=1,...,m_i$.  Take the set of charts $$\phi_{j,i}:B^{k}(0,\beta_{i})\times B^{d-k}(0,\beta_{i})\rightarrow B(p_{j,i},\beta_{i})  \text{ where }\phi_{j,i}=\text{exp}_{p_{j,i}}\circ\tau_{p_{j,i}} ^{-1}.$$ It follows from  Lemmas \ref{limitadocones12} and  \ref{limitadocontract12} that:

\begin{lem} For all $i\in \mathbb{Z}$ and $j=1,...,m_i$: 
\begin{enumerate}[i.]
\item $M_i= \bigcup _{j=1}^{m_{i}}\phi_{j,i}(B^{k}(0,\beta_{i})\times B^{d-k}(0,\beta_{i}))$,
\item $\phi_{j,i+1}^{-1}\textbf{\textit{g}}\phi_{j,i}(B^{k}(0,\beta_{i})\times B^{d-k}(0,\beta_{i}))\subseteq B^{k}(0,\delta_{i+1})\times B^{d-k}(0,\delta_{i+1}).$ 
\item  $ \phi_{j,i}^{-1}\textbf{\textit{g}}^{-1}\phi_{j,i+1}(B^{k}(0,\beta_{i+1})\times B^{d-k}(0,\beta_{i+1}))\subseteq B^{k}(0,\delta_{i})\times B^{d-k}(0,\delta_{i}).$
\item  For all   $v\in B^{k}(0,\beta_{i} )\times B^{d-k}(0,\beta_{i} )$,     if $x\in \overline{K_{\alpha}^{u}},$ we have \[D_{v}(\phi_{j,i+1}^{-1}\textbf{\textit{g}}\phi_{j,i})
(\overline{K_{\alpha}^{u} })
\subseteq K_{\alpha}^{u}\quad \text{and}\quad\Vert D_{v}(\phi_{j,i+1}^{-1}\textbf{\textit{g}}\phi_{j,i})(x)\Vert \geq \eta^{-1}\Vert x \Vert.\]
\item For all   $v\in B^{k}(0,\beta_{i+1} )\times B^{d-k}(0,\beta_{i+1} )$,   if $x\in \overline{K_{\alpha}^{s}}$,  we have \[ D_{v}(\phi_{j,i}^{-1}\textbf{\textit{g}}^{-1}\phi_{j,i+1})(\overline{K_{\alpha}^{s}})
\subseteq K_{\alpha}^{s}\quad\text{and}\quad\Vert   D_{v}(\phi_{j,i}^{-1}\textbf{\textit{g}}^{-1}\phi_{j,i+1})(x)\Vert\geq \eta^{-1}\Vert x\Vert.\] 
\end{enumerate}
\end{lem}

Hence, since $D_{0}\text{exp}_{p}=Id_{T_{p}M}$, $\tilde{g}_{p}=\tau_{f(p)}\circ\text{exp}_{f(p)}^{-1}\circ g_{i}\circ\text{exp}_{p}\circ\tau_{p} ^{-1}$ and $\tau_{p}$ is an isometry,  by choosing $\beta_{i}$  even small, if necessary, we have:

\begin{lem}\label{fred} There exists $\eta\in (0,1)$ such that, if $\textbf{\textit{g}}\in B(\textbf{\textit{f}},(\varepsilon_{i})_{i\in\mathbb{Z}})$,     for each $p\in\textbf{M}$ we have:
\begin{enumerate}[i.]
\item $ D_{p}\textbf{\textit{g}}(K_{\alpha,\textbf{\textit{f}},p}^{u})\subseteq K_{\alpha,\textbf{\textit{f}},\textbf{\textit{g}}(p)}^{u}$. Furthermore,   \(\Vert D_{p}\textbf{\textit{g}}(v )\Vert \geq \eta ^{-1}  \Vert v \Vert \)  if  \(v\in  K_{\alpha,\textbf{\textit{f}},p}^{u}.\)
\item $ D_{\textbf{\textit{g}}(p)}\textbf{\textit{g}}^{-1}(K_{\alpha,\textbf{\textit{f}},\textbf{\textit{g}}(p)}^{s})\subseteq K_{\alpha,\textbf{\textit{f}},p}^{s}$. Furthermore,       \(\Vert D_{\textbf{\textit{g}}(p)}\textbf{\textit{f}}^{-1}(v )\Vert \geq  \eta^{-1}  \Vert v\Vert\) if \(v \in  K_{\alpha,\textbf{\textit{f}},\textbf{\textit{g}}(p)}^{s}.\)
 \end{enumerate}
\end{lem}

\section{Openness of the  Anosov Families}
 
A well-known fact is that the set consisting of  Anosov diffeomorphisms     on a compact Riemannian manifold is open (see, for example, \cite{Shub}). The purpose of this section is to show the result analogous to Anosov families, that is, we prove that $\mathcal{A}(\textbf{M})$ is an open subset of  $\mathcal{F}(\textbf{M})$.   As we have seen in Section 3, the set consisting of   constant families associated to Anosov diffeomorphisms of class   $C^{2} $  is open in    $\mathcal{F}(\textbf{M})$. On the other hand, let $X$ be a compact metric space,  $\phi:X\rightarrow X$ a homeomorphism and $A:X\rightarrow SL(\mathbb{Z},d)$ a continuous map  such that the linear cocycle $F$ defined by $A$ over $\phi$ is hyperbolic. Thus, there exists  $\varepsilon >0$ such that, if     $B:X\rightarrow SL(\mathbb{Z},d)$ is continuous and $\Vert A(x)-B(x)\Vert <\varepsilon$ for all $x\in X$, then the  linear  cocycle $G$ defined by $B$ over $\phi$ is  hyperbolic (see   \cite{Viana}). This fact shows the stability of Anosov families that are obtained by  hyperbolic cocycles.  These are particular cases of our result. 
 
\medskip

First we   prove    the set consisting of Anosov families satisfying the  property of the angles is open and in the end of this work we will show the general case. We will consider $(\varepsilon_{i})_{i\in\mathbb{Z}}$ as in Lemma \ref{fred} and  fix $\textbf{\textit{g}}\in B(\textbf{\textit{f}},(\varepsilon_{i})_{i\in\mathbb{Z}})$.

\begin{lem}\label{invariantesespacios} 
For each  $p\in \textbf{M}$, take
 \begin{equation}\label{familiasespacios} F^{s}_{p}=\bigcap_{n=0}^{\infty}D_{\textbf{\textit{g}}^{n}(p)}\textbf{\textit{g}}^{-n}  (\overline{K_{\alpha,\textbf{\textit{f}},\textbf{\textit{g}}^{n}(p)}^{s}})\quad\text{ and }\quad F^{u}_{p}=\bigcap_{n=0}^{\infty}D_{\textbf{\textit{g}}^{-n}(p)}\textbf{\textit{g}}^{n} (\overline{K_{\alpha,\textbf{\textit{f}},\textbf{\textit{g}}^{-n}(p)}^{u}}).\end{equation}
 Thus, the families  $F^{s}_{p}$ and $F^{u}_{p}$ are $D\textbf{\textit{g}}$-invariant. 
 (see Figure \ref{intercone}).   
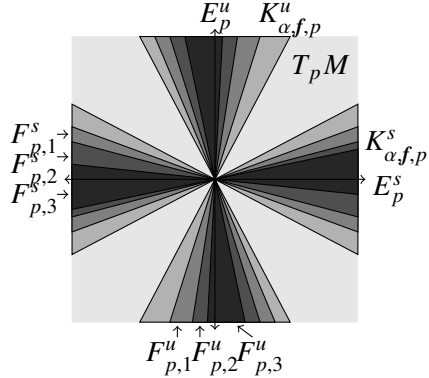
\begin{figure}[ht] 
\begin{center}

\begin{tikzpicture}
\draw[black!7,fill=black!10, ultra thin] (-7.9,-0.4)rectangle (-4.1,3.4);
\draw[black, fill=black!30, thin] (-6,1.5) -- (-7,3.4) -- (-5,3.4) -- cycle;
\draw[black, fill=black!50, thin] (-6,1.5) -- (-6.8,3.4) -- (-5.4,3.4) -- cycle;
\draw[black, fill=black!69, thin] (-6,1.5) -- (-6.6,3.4) -- (-5.7,3.4) -- cycle;
\draw[black, fill=black!85, thin] (-6,1.5) -- (-6.4,3.4) -- (-5.9,3.4) -- cycle;
\draw[black, fill=black!30, thin] (-6,1.5) -- (-7,-0.4) -- (-5,-0.4) -- cycle;
\draw[black, fill=black!50, thin] (-6,1.5) -- (-5.2,-0.4) -- (-6.6,-0.4) -- cycle;
\draw[black, fill=black!69, thin] (-6,1.5) -- (-6.3,-0.4) -- (-5.4,-0.4) -- cycle;
\draw[black, fill=black!85, thin] (-6,1.5) -- (-6.1,-0.4) -- (-5.6,-0.4) -- cycle;

\draw[black, fill=black!30, thin] (-6,1.5) -- (-4.1,2.5) -- (-4.1,0.5) -- cycle;
\draw[black, fill=black!50, thin] (-6,1.5) -- (-4.1,2.2) -- (-4.1,0.8) -- cycle;
\draw[black, fill=black!69, thin] (-6,1.5) -- (-4.1,2) -- (-4.1,1) -- cycle;
\draw[black, fill=black!85, thin] (-6,1.5) -- (-4.1,1.9) -- (-4.1,1.3) -- cycle;

\draw[black, fill=black!30, thin] (-6,1.5) -- (-7.9,2.5) -- (-7.9,0.5) -- cycle;
\draw[black, fill=black!50, thin] (-6,1.5) -- (-7.9,2.2) -- (-7.9,0.8) -- cycle;
\draw[black, fill=black!69, thin] (-6,1.5) -- (-7.9,2) -- (-7.9,1) -- cycle;
\draw[black, fill=black!85, thin] (-6,1.5) -- (-7.9,1.7) -- (-7.9,1.1) -- cycle;

 \draw[<->] (-6,-0.5) -- (-6,3.5);
 \draw[<->] (-8,1.5) -- (-4,1.5); 
\draw (-6.2,4) node[below] {\quad{\small $E_{p}^{u}$}}; 
\draw (-3.9,1.7) node[below] {\quad{\small $E_{p}^{s}$}};
\draw (-4.6,3.3) node[below] {\small $T_{p}M$};
\draw (-5.4,-0.5) node[below] {\small $F_{p,3}^{u}$};
 \draw[->] (-5.5,-0.6) -- (-5.7,-0.45);
\draw (-6,-0.5) node[below] {\small $F_{p,2}^{u}$};
\draw[->] (-6.2,-0.6) -- (-6.2,-0.45);
\draw (-6.6,-0.5) node[below] {\small $F_{p,1}^{u}$};
\draw[->] (-6.5,-0.6) -- (-6.5,-0.45);
\draw (-8.4,1.6) node[below] {\small $F_{p,3}^{s}$}; 
\draw[->] (-8.1,2.1) -- (-7.95,2.1);
\draw (-8.4,2) node[below] {\small $F_{p,2}^{s}$};
\draw[->] (-8.1,1.8) -- (-7.95,1.8);
\draw (-8.4,2.4) node[below] {\small $F_{p,1}^{s}$};
\draw[->] (-8.1,1.3) -- (-7.95,1.3);

\draw (-5,4) node[below] {\small $K_{\alpha,\textbf{\textit{f}},p}^{u}$};
\draw (-3.8,2.3) node[below] {\quad{\small $K_{\alpha,\textbf{\textit{f}},p}^{s}$}};
 \end{tikzpicture} 

\end{center}
\caption{$F_{p,n}^{r}=\bigcap_{k=1}^{n}D \textbf{\textit{g}}^{\pm k}_{\textbf{\textit{g}}^{\pm k}(p)}  (\overline{K_{\alpha,\textbf{\textit{f}},\textbf{\textit{g}}^{\pm k}(p)}^{s}})$, for $r=s,u$ and $n=1,2,3$.} \label{intercone}
\end{figure}
\end{lem}
\begin{proof} By Lemma \ref{limitadocones12} we have  for all   $p\in \textbf{M}$,   
$D_{\textbf{\textit{g}}(p)}\textbf{\textit{g}}^{-1}(\overline{K_{\alpha,\textbf{\textit{f}},\textbf{\textit{g}}(p)}^{s}})
\subseteq K_{\alpha,\textbf{\textit{f}},p}^{s} $  and $ D_{p}g
(\overline{K_{\alpha,\textbf{\textit{f}},p}^{u}})
\subseteq K_{\alpha,\textbf{\textit{f}},\textbf{\textit{g}}(p)}^{u}
.$ 
Thus $D_{\textbf{\textit{g}}(p)}\textbf{\textit{g}}^{-1} (F^{s}_{\textbf{\textit{g}}(p)}) 
 \subseteq\bigcap_{n=0}^{\infty}
D_{\textbf{\textit{g}}^{n}(p)}\textbf{\textit{g}}^{-n}(\overline{K_{\alpha,\textbf{\textit{f}},\textbf{\textit{g}}^{n}(p)}^{s}})
 =F_{p}^{s}.$
On the other hand, \begin{align*}D_{p}\textbf{\textit{g}}(F^{s}_{p})&=
D_{p}\textbf{\textit{g}} (\overline{K_{\alpha,\textbf{\textit{f}},p}^{s}})
\cap\bigcap_{n=1}^{\infty}
D_{p}g  (D_{\textbf{\textit{g}}^{n}(p)}\textbf{\textit{g}}^{-n} (\overline{K_{\alpha,\textbf{\textit{f}},\textbf{\textit{g}}^{n}(p)}^{s}}))
\\
&\subseteq   \bigcap_{n=0}^{\infty}
D_{\textbf{\textit{g}}^{n+1}(p)}  \textbf{\textit{g}}^{-n} (\overline{K_{\alpha,\textbf{\textit{f}},\textbf{\textit{g}}^{n+1}(p)}^{s}})  =F^{s}_{\textbf{\textit{g}}(p)}.
\end{align*}
Consequently,  $D_{p}\textbf{\textit{g}} (F^{s}_{p})=F^{s}_{\textbf{\textit{g}}(p)}.$
Analogously we can prove     $D_{p}\textbf{\textit{g}}(F^{u}_{p})=F^{u}_{\textbf{\textit{g}}(p)}.$
\end{proof}

 Inductively we have    
$D_{p}\textbf{\textit{g}}^{n} (F^{s}_{p}) = F_{\textbf{\textit{g}}^{n}(p)}^{s} $ and $D_{p}\textbf{\textit{g}}^{n} (F^{u}_{p}) = F_{\textbf{\textit{g}}^{n}(p)}^{u},$   for all $n\geq1.$
 Since  $F_{p}^{r}\subseteq K^{r} _{\alpha,\textbf{\textit{f}},p}$ for  $r=s,u$,  it follows from  Lemma \ref{fred} that, for all $n\geq1$,   \[\Vert D_{p}\textbf{\textit{g}}^{n}v\Vert \geq \frac{1}{\eta^{n}}\Vert v\Vert\text{ for }v\in F^{u}_{p}\quad\text{and} \quad\Vert D_{p}\textbf{\textit{g}}^{-n} v\Vert \geq \frac{1}{\eta^{n}}\Vert v\Vert\text{ for } v\in F^{s}_{p}  .\]

\begin{lem}\label{ultimolema}   $F_{p}^{s}$ and  $F_{p}^{u}$  given in  \eqref{familiasespacios} are vectorial subspaces and  furthermore  $T_{p}\textbf{M}=F^{s}_{p}\oplus F^{u}_{p}$, for each $p\in\textbf{M}$.  
\end{lem}
\begin{proof} See    Proposition 7.3.3 in \cite{luisb}.
\end{proof}

\begin{propo}\label{propoaberang}
   \textbf{\textit{g}} is an    Anosov family   and satisfies the property of the angles.
\end{propo}
\begin{proof} From Lemmas \ref{fred}, \ref{invariantesespacios} and \ref{ultimolema}  we have that, considering the splitting  $T_{p}\textbf{M}=F^{s}_{p}\oplus F^{u}_{p}$, for each $p\in\textbf{M}$, \textbf{\textit{g}} has  hyperbolic behavior.  We  can prove that this splitting is unique (see Lemma \ref{unicidadesubs}) and depends continuously on $p$ (see Proposition \ref{continuidade}).  Consequently,  \textbf{\textit{g}} is an Anosov family. Finally, since  
\( F^{s}_{p}\subseteq K^{s}_{\alpha,\textbf{\textit{f}},p}\)   and \( F^{u}_{p}\subseteq K^{u}_{\alpha,\textbf{\textit{f}},p}\) for all $p$ and $ \alpha <\frac{1- \lambda}{1+ \lambda}<1,$ we have that $\textbf{\textit{g}} $   s.   p.   a.  
\end{proof}

From    Proposition \ref{propoaberang} we obtain the set consisting of Anosov families   that  s.   p. a. is open in  $\mathcal{F}(\textbf{M})$.  Finally will show that the set consisting of all the Anosov families is open in  $\mathcal{F}(\textbf{M}).$   In order to prove this result, let's see the following facts: suppose that   $(\textbf{M},\langle\cdot,\cdot\rangle, \textbf{\textit{f}})$   does not s. p. a. with the  Riemannian metric  $\langle\cdot,\cdot\rangle$. 
Thus $(\textbf{M},\langle\cdot,\cdot\rangle^{\star}, \textbf{\textit{f}})$ is a strictly   Anosov family  that  s.   p.   a.  with the   Riemannian metric 
$\langle\cdot,\cdot\rangle^{\star}$ obtained  in   Proposition   \ref{mather}.   
Fix $\varepsilon >0$ and take $\Delta_{i}=\frac{1}{\mu_{i}}(\frac{\lambda+\varepsilon}{\varepsilon}c)^{2}$
 (see   \eqref{norma231}).
 Thus,  \[ \Delta_{i}^{-1}\Vert v\Vert^{\star}\leq  \Vert v\Vert \leq 2\Vert v\Vert ^{\star}\quad\text{  for all }v\in TM_{i}, i\in\mathbb{Z},\] 
where $\Vert \cdot\Vert $ and  $\Vert \cdot\Vert ^{\star}$ are the norms induced by  $\langle\cdot,\cdot\rangle$ and $\langle\cdot,\cdot\rangle^{\star}$ on $\textbf{M}$, respectively.
  From   Proposition \ref{propoaberang} it follows that there exists  a sequence  $(\varepsilon_{i})_{i\in\mathbb{Z}}$ such that, if  $\textbf{\textit{g}}=(g_{i})_{i\in\mathbb{Z}}$ is a  non-stationary dynamical system  with $d_{\textbf{D}_{i}}^{\star}(f_{i},g_{i})<\varepsilon_{i}$, then $(\textbf{M},\langle\cdot,\cdot\rangle^{\star},\textbf{\textit{g}})$ is an    Anosov family, where $d_{\textbf{D}_{i}}^{\star}$ is the metric  on   $\textbf{D}_{i}$ induced  by the  metric $\langle\cdot,\cdot\rangle^{\star}$ on $\textbf{M}$. We want to show that each  family in some strong basic neighborhood of \textbf{\textit{f}}    is an Anosov family with the metric $\langle\cdot,\cdot\rangle$.  This fact is not  immediate, since $\langle\cdot,\cdot\rangle$ and $\langle\cdot,\cdot\rangle^{\star}$ are not necessarily uniformly equivalent on $\textbf{M}$ and the notion of Anosov family depends on the metric on the total space.    
  
\begin{teo}\label{teoremaprin}    $\mathcal{A}(\textbf{M})$ is open in   
$\mathcal{F}(\textbf{M})$. 
\end{teo}
\begin{proof} 
 If \textbf{\textit{f}} satisfies the property of the  angles, by Proposition \ref{propoaberang}   there exists a strong basic neighborhood $B(\textbf{\textit{f}},( \varepsilon_{i})_{i\in\mathbb{Z}})$ of \textbf{\textit{f}} such that, if  $\textbf{\textit{g}}\in B(\textbf{\textit{f}},(  \varepsilon _{i})_{i\in\mathbb{Z}})$   then \textbf{\textit{g}} is  an   Anosov family.    Suppose that  \textbf{\textit{f}} does not  satisfy the property of the angles. From  Proposition  \ref{propoaberang}  we have there exists a  sequence of   positive   numbers  $ (\varepsilon_{i})_{i\in\mathbb{Z}}$ such that, if  $\textbf{\textit{g}}=(g_{i})_{i\in\mathbb{Z}}\in \mathcal{F}(\textbf{M})$  and  $d_{\textbf{D}_{i}}^{\star}(f_{i},g_{i})<\varepsilon_{i}$, then $(\textbf{M},\langle\cdot,\cdot\rangle^{\star},\textbf{\textit{g}})$ is  a  strictly Anosov family with  constant $\tilde{\lambda}=\eta \in(0,1)$. For each $i$, take   $\tilde{\varepsilon}_{i} =\varepsilon_{i}/\Delta_{i}.$  Notice that if      $d_{\textbf{D}_{i}}(f_{i},g_{i})<\tilde{\varepsilon}_{i} $ then  $d_{\textbf{D}_{i}^{\star}}(f_{i},g_{i})<\varepsilon_{i}$, for all $i$. Consequently, if $\textbf{\textit{g}}\in B(\textbf{\textit{f}},(\tilde{\varepsilon}_{i})_{i\in\mathbb{Z}})$, then $(\textbf{M},\langle\cdot,\cdot\rangle^{\star},\textbf{\textit{g}})$ is an  Anosov family. Consider the stable subspace   $E^{s}_{\textbf{\textit{g}},p} $ of  \textbf{\textit{g}} at $p$ (with respect to  the metric $\langle\cdot,\cdot\rangle^{\star}$). If $v\in E^{s}_{\textbf{\textit{g}},p} $, then  $v=v_{s}+v_{u}$, where  $v_{s}\in E^{s}_{\textbf{\textit{f}},p}$ and  $v_{u}\in E^{u}_{\textbf{\textit{f}},p}$.
  Take      $\alpha\in (0,N)$, where $N=\min\{\frac{\varepsilon}{c(\lambda+\varepsilon)},\frac{1-\lambda}{1+\lambda}\}$.   Since the stable    subspaces    of  \textbf{\textit{g}} are contained in the  stable  $\alpha$-cones of \textbf{\textit{f}} and $\Vert v_{s}\Vert \leq \Vert v_{s}\Vert^{\star}$, it follows from     \eqref{qwerty} that 
\[\Vert v_{s}\Vert \leq  \Vert v_{s}+v_{u}\Vert + \Vert v_{u}\Vert \leq \Vert v_{s}+v_{u}\Vert +\alpha\Vert v_{s}\Vert^{\star}\leq \Vert v \Vert+\alpha \frac{\lambda+\varepsilon}{\varepsilon}c\Vert v_{s}\Vert.\]
   Thus $(1- \alpha \frac{\lambda+\varepsilon}{\varepsilon}c)\Vert v_{s}\Vert \leq \Vert v\Vert $ (note that $1- \alpha \frac{\lambda+\varepsilon}{\varepsilon}c>0$ because $\alpha<\frac{\varepsilon}{c(\lambda+\varepsilon)}$).  Hence
\begin{align*}\Vert D_{p}g^{n} (v)\Vert &\leq 2\Vert D_{p}g^{n} (v)\Vert^{\star} \leq 2\eta^{n}(\Vert v_{s}\Vert^{\star} +\Vert v_{u}\Vert^{\star}) 
  \leq   2\eta^{n}(1+\alpha)\Vert v_{s}\Vert^{\star}\\
& \leq 2\eta^{n}(1+\alpha)\frac{\lambda+\varepsilon}{\varepsilon}c (1-\alpha \frac{\lambda+\varepsilon}{\varepsilon}c)^{-1}\Vert v\Vert= c^{\prime}\eta^{n}\Vert v\Vert,
\end{align*}
where $c^{\prime}=2(1+\alpha)\frac{\lambda+\varepsilon}{\varepsilon}c (1-\alpha \frac{\lambda+\varepsilon}{\varepsilon}c)^{-1}.$ Analogously we have     $\Vert D_{p}\textbf{\textit{g}}^{-n} (v)\Vert\leq c^{\prime}\eta^{n}\Vert v\Vert$ for $v\in E^{u}_{\textbf{\textit{g}},p} $. 
Hence,   $(\textbf{M},\langle\cdot,\cdot\rangle, \textbf{\textit{g}})$ is an Anosov family with  constants    $\eta $ and $c^{\prime}$.  
 \end{proof}

 Note that for the basic strong neighborhoods   $B(\textbf{\textit{f}},(\varepsilon_{i})_{i\in\mathbb{Z}})$ of a  system $(f_{i})_{i\in\mathbb{Z}}$ the  $\varepsilon_{i}$ can be arbitrarilly small for   $|i|$ large.  When there exists   $\varepsilon >0$ such that  $\varepsilon_{i}=\varepsilon$ for all $i\in\mathbb{Z}$,   the neighborhood is called \textit{uniform}. As noted above, when    \textbf{\textit{f}} is the constant family     associated to an Anosov diffeomorphism,   it is possible to find  a   uniform neighborhood of   \textbf{\textit{f}} whose elements are  Anosov families.  In general it is not possible to find  a   uniform neighborhood of an Anosov family such that each family in that neighborhood is Anosov. For example,    if the angles between the stable   and unstable subspace   decay, or if we can not get the inequality \eqref{sigmaadecuado}  with a   uniform $\beta_{i}$, etc., it is necessary to take the  $\varepsilon_{i}$'s  ever smaller. In \cite{Jeo3} we will give conditions on the families for obtain uniform neighborhoods.

\medskip
 
\noindent  \textbf{Acknowledgements.} The present work was carried out with the support of   the Conselho Nacional de Desenvolvimento Cient\'ifico  e Tecnol\'ogico - Brasil (CNPq) and the Coordena\c c\~ao de Aperfei\c coamento de Pessoal de N\'ivel Superior  (CAPES).  The author would like to thank the institutions  Universidade de S\~ao Paulo (USP) and Instituto de Matem\'atica Pura e Aplicada (IMPA) for their   hospitality    during the course of the writing. Special thanks for A. Fisher, my doctoral advisor, who has inspired and aided me along the way.

\end{document}